\documentclass[reqno,11pt]{amsart}
\usepackage{amsmath,amssymb,latexsym,soul,cite,mathrsfs}
\usepackage{color,enumitem,graphicx}
\usepackage{amsmath,amssymb,latexsym,soul,cite,mathrsfs}
\usepackage{color,enumitem,graphicx}
\usepackage[colorlinks=true,urlcolor=blue,
citecolor=red,linkcolor=blue,linktocpage,pdfpagelabels,
bookmarksnumbered,bookmarksopen]{hyperref}
\usepackage[english]{babel}
\usepackage{lineno}

\usepackage[left=2.4cm,right=2.8cm,top=2.5cm,bottom=2.4cm]{geometry}

\usepackage[hyperpageref]{backref}

\usepackage{marginnote}
\usepackage[hyperpageref]{backref}
\usepackage{hyperref}
\usepackage{amssymb,amsmath,latexsym,soul,cite,mathrsfs}
\usepackage[english=usenglishmax]{hyphsubst}
\usepackage{lineno}
\usepackage{euscript}

\numberwithin{equation}{section}
\newtheorem{theorem}{Theorem}[section]
\newtheorem{lemma}[theorem]{Lemma}

\newtheorem{corollary}[theorem]{Corollary}

\allowdisplaybreaks

\title[Ground state solutions for bi-harmonic
Kirchhoff-type problems]{Nehari-type ground state solutions for
asymptotically periodic bi-harmonic
Kirchhoff-type problems in $\mathbb{R}^N$}

\author[Souza Filho]{A. P. F. Souza Filho}
\address{Departamento de Ci\^encias Exatas e Naturais, Universidade Federal Rural do Semi-\'Arido, 59900-000, Pau dos Ferros--RN, Brazil}
\email{padua.filho@ufersa.edu.br}



\thanks{}
 
 \thanks{}

\begin{document}

%%%%%%%%%%%%%%%%%%%%%%%%%%%%%%%%%%%%%%%%%%%%%%%%%%%%%%%%%%%%%%%%%%%%%%%%%%%%%%%%%%%%%%%%%%%%%%%%%%%%%%%%%%%%%%%%%%%%%%%%%%%%%%%%%%%%
%
%                               ABSTRACT
%
%%%%%%%%%%%%%%%%%%%%%%%%%%%%%%%%%%%%%%%%%%%%%%%%%%%%%%%%%%%%%%%%%%%%%%%%%%%%%%%%%%%%%%%%%%%%%%%%%%%%%%%%%%%%%%%%%%%%%%%%%%%%%%%%%%%%

\begin{abstract}
We investigate the following Kirchhoff-type  biharmonic equation 
\begin{equation}\label{pr}
\left\{
\begin{array}{ll}
\Delta^2 u+ \left(a+b\int_{\mathbb{R}^N}|\nabla u|^2d x\right)(-\Delta u+V(x)u)=f(x,u),\quad x\in \mathbb{R}^N,\\
u\in H^{2}(\mathbb{R}^N),
\end{array}
\right.
\end{equation}
where $a>0$, $b\geq 0$ and $V(x)$ and $f(x, u)$ are periodic or asymptotically periodic in $x$. We study the existence of Nehari-type ground state solutions of the problem just above with $f(x,u)u-4F(x,u)$ sign-changing, where $F(x,u):=\int_0^uf(x,s)d s$. We significantly extend some results from the previous literature.
\end{abstract}

\maketitle

%%%%%%%%%%%%%%%%%%%%%%%%%%%%%%%%%%%%%%

%\bigskip
%\begin{center}
%\begin{minipage}{8cm}
%\footnotesize
%\tableofcontents
%\end{minipage}
%\end{center}

%%%%%%%%%%%%%%%%%%%%%%%%%%%%%%%%%%%%%%
%\bigskip

%%%%%%%%%%%%%%%%%%%%%%%%%%%%%%%%%%%%%%%%%%%%%%%%%%%%%%%%%%%%%%%%%%%%%%%%%%%%%%%%%%%%%%%%%%%%%%%%%%%%%%%%%%%%%%%%%%%%%%%%%%%%%%%%%%%%%
%
%                                               INTRODUCTION
%
%%%%%%%%%%%%%%%%%%%%%%%%%%%%%%%%%%%%%%%%%%%%%%%%%%%%%%%%%%%%%%%%%%%%%%%%%%%%%%%%%%%%%%%%%%%%%%%%%%%%%%%%%%%%%%%%%%%%%%%%%%%%%%%%%%%%

\section{INTRODUCTION AND STATEMENT OF MAIN RESULTS}
\noindent
This paper is concerned with the following  Kirchhoff type problem:
\begin{equation}\label{p}
\left\{
\begin{array}{ll}
\Delta^2 u+ \left(a+b\int_{\mathbb{R}^N}|\nabla u|^2d x\right)(-\Delta u+V(x)u)=f(x,u),\quad x\in \mathbb{R}^N,\\
u\in H^{2}(\mathbb{R}^N),
\end{array}
\right.
\end{equation}

where $a>0$, $b\geq 0$, $N\geq 5$, $V(x)\in\mathcal{C}(\mathbb{R}^N,\mathbb{R})$  and $f \in\mathcal{C}(\mathbb{R}^N\times \mathbb{R} ,\mathbb{R})$ is a function with a subcritical growth.

	In the recent years, bi-harmonic and non-local operators arise in the description of various phenomena in the pure mathematical research and concrete real-world applications, for example, for studying the traveling waves in suspension bridges (see \cite{lmc1,lmc2}) and describing the static deflection of an elastic plate in fluid (see \cite{plate}). Problem \eqref{p} is called a non-local problem because of the presence of the term $b\int_{\mathbb{R}^N}|\nabla u|^2d x$ which indicates
that \eqref{p} is not a pointwise identity. This causes some
mathematical difficulties which makes the study of \eqref{p} particularly interesting.

Note that if we consider $a=1$ and $b=0$ the fourth-order elliptic equation of Kirchhoff type above corresponds to the following nonlinear Schr\"{o}dinger equation in $\mathbb{R}^N$ ($N\geq 5$):
\begin{equation*}
\left\{
\begin{array}{ll}
\Delta^2 u - \Delta u+V(x)u=f(x,u),\quad x\in \mathbb{R}^N,\\
u\in H^{2}(\mathbb{R}^N).
\end{array}
\right.
\end{equation*}
This class of nonlinear elliptic equations in $\mathbb{R}^N$ has been studied  by many authors in literature motivated by mathematical and physical problems in particular to studying the standing wave solutions. Some important related results for bi-harmonic equations the interested reader is referred to are \cite{jm, decay, regularity, Noussair, mramos} and the references therein. On the other hand, problem \eqref{p} is related to the stationary analogue of the Kirchhoff equation
$$
u_{t t}-\left(a+b \int_{\Omega}|\nabla u|^{2}d x\right) \Delta u=f(x, u),
$$
where $\Omega \subset \mathbb{R}^{N}$ is a smooth bounded domain, which was proposed by Kirchhoff in \cite{k} as an extension of the classical D'Alembert's wave equation for free vibrations of elastic strings.

In the last years, many researchers have studied several questions about the following Kirchhoff-type elliptic equation
\begin{equation}\label{b1.3}
\begin{array}{lr}-\left(a+b \int_{\Omega}|\nabla u|^{2} d x\right) \Delta u=f(x, u), & {x \in \Omega}.  
\end{array}.
\end{equation}
where $\Omega$ is a domain in $\mathbb{R}^N.$ For instance, results on the existence and multiplicity of nontrivial solutions for \eqref{b1.3} have been established when $\Omega$ is bounded and $u=0$ on $\partial \Omega$, see for instance, \cite{b7,b8,b9,b10} and the references therein. Recently, many  authors studied the existence and multiplicity
of nontrivial solutions of \\
\begin{equation}
\left\{\begin{array}{l}{-\left(a+b \int_{\mathbb{R}^{N}}|\nabla u|^{2} d x\right) \Delta u+V(x) u=f(x, u), x \in \mathbb{R}^{N} ;} \\ 
{u \in H^{1}\left(\mathbb{R}^{N}\right)\quad (N=1,2,3),}\end{array}\right.
\end{equation}
see for example, \cite{b19,20imb,b21}.

Inspired by the works of \cite{bitao, bcub,bcubfrac}, a natural question is whether the same results occurs for the following Kirchhoff-type  biharmonic equation 
\begin{equation*}
\left\{
\begin{array}{ll}
\Delta^2 u+ \left(a+b\int_{\mathbb{R}^N}|\nabla u|^2 d x\right)(-\Delta u+V(x)u)=f(x,u),\quad x\in \mathbb{R}^N,\\
u\in H^{2}(\mathbb{R}^N),
\end{array}
\right.
\end{equation*}

where $a>0$, $b\geq 0$, $N\geq 5$, $V(x)\in\mathcal{C}(\mathbb{R}^N,\mathbb{R})$  and $f \in\mathcal{C}(\mathbb{R}^N\times \mathbb{R} ,\mathbb{R})$ satisfy the following hypotheses.

We now formulate assumptions for $V$ and $f$ in problem~\eqref{p}.
\vskip4pt
\noindent
$\bullet$ {\sc Assumptions on $V$.} 
\begin{description}
	\item[$(V)$]  $V \in \mathcal{C}\left(\mathbb{R}^{N},(0, \infty)\right)$ is 1-periodic in each of $x_1$,$x_2$,...,$x_N$ and $\inf_{\mathbb{R}^N}V>0$.

	\item[ $(V')$] $V(x)=V_{0}(x)+V_{1}(x), V_{0}, V_{1} \in \mathcal{C}\left(\mathbb{R}^{N}, \mathbb{R}\right), V_{0}(x)$ is 1 -periodic in $x_{1}, x_{2},...,x_{N},$ $V_{1}(x) \leq 0$ for $x \in \mathbb{R}^{N}$ and $V_{1} \in \mathcal{B},$ where $\mathcal{B}$ be the class of functions $b \in \mathcal{C}\left(\mathbb{R}^{N}\right) \cap L^{\infty}\left(\mathbb{R}^{N}\right)$ such that for every $\epsilon>0,$ the set $\left\{x \in \mathbb{R}^{N} :|b(x)| \geq \epsilon\right\}$ has finite Lebesgue measure;
\end{description}
\vskip4pt
\noindent
$\bullet$ {\sc Assumptions on $f$.} 
\begin{description}
	\item[$(f1)$] $ f(x,u)$ is 1-periodic in each of $x_1$,$x_2$,...,$x_N$ for all $u\in\mathbb{R}$ and there exist constants $C> 0$ and $p \in (4,2_*)$, where $2_*=2N/(N-4),$ such that 
	\[
	|f(x,u)|\leq C(1+|u|^{p-1}), \quad \mbox{for all}\quad (x,u)\in \mathbb{R}^N\times\mathbb{R};
  \]
  
 	\item[$(f2)$] $f(u)=o(|u|)$ uniformly in $x$ as $|u|\to 0$;
	
	\item[$(f3)$] 
		\[
		\lim_{|u|\to\infty}\frac{f(x,u)}{u^3}=\infty,\quad \mbox{uniformly in}\quad x;
		\]
	\item[$(f4)$] there exists $\mu \in(0,1)$ such that for any $t>0$ and $u \in \mathbb{R} \backslash\{0\}$
	$$
\left[\frac{f(x, u)}{u^{3}}-\frac{f(x, t u)}{(t u)^{3}}\right] \operatorname{sign}(1-t)+\mu aV(x)\frac{\left|1-t^{2}\right|}{(t u)^{2}} \geq 0.
$$

 \item[$(f5)$] $f(x, t)=f_{0}(x, t)+f_{1}(x, t), f_{0} \in \mathcal{C}\left(\mathbb{R}^{N} \times \mathbb{R}, \mathbb{R}\right), f_{0}(x, t)$ is 1 -periodic in $x_{1}, x_{2},...x_{N}$ and for any $x \in \mathbb{R}^{N}, t>0$ and $f_{1} \in \mathcal{C}\left(\mathbb{R}^{N} \times \mathbb{R}, \mathbb{R}\right),$ satisfies
  \begin{equation}\label{1.6}
\left|f_{1}(x, t)\right| \leq h(x)\left(|t|+|t|^{q-1}\right), \quad f_{1}(x, t) t \geq 0
\end{equation}
	where $F_{1}(x, t)=\int_{0}^{t} f_{1}(x, s) \mathrm{d} s, q \in\left(2,2_{*}\right)$ and $h \in \mathcal{B}.$

\item[$(f6)$] there exists $\mu \in(0,1)$ such that for any $t>0$ and $u \in \mathbb{R} \backslash\{0\}$
\begin{equation*}\label{1.5}
\left[\frac{f_{0}(x, u)}{u^{3}}-\frac{f_{0}(x, t \tau)}{(t u)^{3}}\right] \operatorname{sign}(1-t)+\mu aV_{0}(x) \frac{\left|1-t^{2}\right|}{(t u)^{2}} \geq 0.
\end{equation*}
\end{description}

\noindent
Now, let us introduce some notations.
Let
\[
H:=H^2(\mathbb{R}^N):=\left\{u\in L^{2}(\mathbb{R}^N): |\nabla u|, \Delta u\in   L^{2}(\mathbb{R}^N)\right\}
\]
endowed with the norm
\[
\|u\|=\Big(\int_{\mathbb{R}^N}(|\Delta u|^2+a(|\nabla u|^2+V(x)u^{2}) )d x\Big)^{1/2}
\]
and the inner product
\[
(u,v)=\int_{\mathbb{R}^N}(\Delta u \Delta v+a\nabla u \nabla v+aV(x)uv )d x.
\]

We consider the following energy functional

\begin{equation}\label{funcj}
J(u)=\frac{1}{2}\int_{\mathbb{R}^N}(|\Delta u|^2+a(|\nabla u|^2+V(x)u^2))d x+\frac{b}{4}(|\nabla u|^4_2)+\frac{b}{2}(|\nabla u|^2_2)\int_{\mathbb{R}^N}V(x)u^2d x-\int_{\mathbb{R}^N}F(x,u)d x
\end{equation}
for all $u\in E$. We can see that $J$ is well defined on $H$ and $J\in C^1(H, \mathbb{R})$ and its Gateaux
derivate is given by
\begin{align}\label{d2.2}
J'(u)v=(u,v)+b\left(|\nabla u|^2_2+\int_{\mathbb{R}^N}V(x)u^2d x\right)\int_{\mathbb{R}^N}\nabla u \nabla v d x  +b(|\nabla u|^2_2)\int_{\mathbb{R}^N}V(x)uvd x -\int_{\mathbb{R}^N}f(x,u)vd x,
\end{align}
for all $u$, $v$ in $H.$\\
\noindent
Now we can state our main result. In the periodic case, we establish the following theorem:
\begin{theorem}\label{thm1.1}
	Assume that $(V)$ and $(f1)$-$(f4)$ are satisfied.  Then problem
	\eqref{p} has a nontrivial solution $u\in \mathcal{N}$ such that $J(u)=\inf_{\mathcal{N}}J>0$, where 
	\begin{align}\label{1.5.1}
	\mathcal{N}:=\{u\in H:u \not =0, J'(u)u=0\}.
	\end{align}
\end{theorem}
\noindent
The next theorem gives a answer when we are in the asymptotically periodic case.
\begin{theorem}\label{thm1.2}
	Assume that $(V')$ and $(f5)$-$(f6)$ are satisfied.  Then problem
	\eqref{p} has a nontrivial solution $u\in \mathcal{N}$ such that $J(u)=\inf_{\mathcal{N}}J>0$, where 
	\begin{align}\label{1.5.2}
	\mathcal{N}:=\{u\in H:u \not =0, J'(u)u=0\}.
	\end{align}
\end{theorem}

\begin{lemma}\label{lemma2.1}
	Assume that $(f1)$-$(f4)$ hold. Then for any $u \in H^2(\mathbb{R}^N )$,
	\begin{equation}\label{2.3}
	J(u)\geq J(tu)+\frac{1-t^4}{4}J'(u)u+(1-\mu)\frac{(1-t^2)^2}{4}\|u\|^2,\quad t\geq 0.
	\end{equation}
\end{lemma}
\begin{proof}
	For any $x\in \mathbb{R}^N$ and $s\in \mathbb{R}^+$, using $(f4)$, for all $t\geq 0$, we have 
	\begin{align}\label{2.4}
	0\leq&  \int_{\tau}^{1}\left(\frac{f(x,t)}{t^3}+\frac{f(x,st)}{(st)^3}+\frac{\mu a V(x)(1-s^2)}{st}^2\right)t^4s^3 ds \nonumber\\
	=& \frac{1-t^4}{4}tf(x,t)-(F(x,t)-F(x,\tau t))+\frac{(1-\tau^2)^2}{4}t^2\mu a V(x).
	\end{align}
	Then, for all $u\in H$, we obtain
	\begin{align*}
	 J(u)-J(tu) = &\frac{1-t^2}{2}\|u\|^2+\frac{b(1-t^4)}{4}|\nabla u|^4_2+\frac{1-t^4}{4}2b|\nabla u|^2_2\int_{\mathbb{R}^N}V(x)u^2d x-\int_{\mathbb{R}^N}(F(x,u)-F(x,tu))d x\\
	 =&\frac{1-t^4}{4}\|u\|^2+\frac{b(1-t^4)}{4}|\nabla u|^4_2+\frac{1-t^4}{4}2b|\nabla u|^2_2\int_{\mathbb{R}^N}V(x)u^2d x+\frac{(1-t^2)^2}{4}\|u\|^2\\
	 &\quad-\int_{\mathbb{R}^N}(F(x,u)-F(x,tu))d x\\
	 =& \frac{1-t^4}{4}J'(u)u+\frac{1-t^4}{4}\int_{\mathbb{R}^N}f(x,u)ud x+\frac{(1-t^2)^2}{4}\|u\|^2-\int_{\mathbb{R}^N}(F(x,u)-F(x,tu))d x\\
	 \geq &\frac{1-t^4}{4}J'(u)u+\frac{(1-t^2)^2}{4}\|u\|^2-\int_{\mathbb{R}^N} \frac{(1-t^2)^2}{4}\mu a V(x)u^2d x\\
	 \geq &\frac{1-t^4}{4}J'(u)u+(1-\mu)\frac{(1-t^2)^2}{4}\|u\|^2,\quad t\geq 0.
	\end{align*}
\end{proof}

\begin{corollary}
	Assume that $(f1)$-$(f4)$ are satisfied. Then, if $u\in \mathcal{N}$, we obtain
		\begin{equation}\label{2.5}
	J(u)\geq J(tu)+(1-\mu)\frac{(1-t^2)^2}{4}\|u\|^2,\quad t\geq 0.
	\end{equation}
\end{corollary}

\begin{corollary}\label{2.2}
		Assume that $(f1)$-$(f4)$ are satisfied. Then, if $u\in \mathcal{N}$, we obtain
	\begin{equation}\label{2.6}
	J(u)=\max_{t\geq 0}J(tu).
	\end{equation}
\end{corollary}

\begin{lemma}
	Assume that $(f1)$-$(f4)$ are satisfied. Then, for any $s\in \mathbb{R}$ and $ x\in \mathbb{R}^N$ ,
	\begin{equation}\label{2.14}
	0\leq \frac{1}{4}f(x,s)s-F(x,s)+\frac{\mu a V(x)}{4}s^2.
	\end{equation}
\end{lemma}

\begin{proof}
	It is enough to take $t = 0$ in \eqref{2.4}.
\end{proof}

\begin{lemma}\label{lemma2.4}
Assume that $(f1)$-$(f4)$ are satisfied. Then, if $u \in H\setminus \{0\}$, there exists unique $t_u > 0$ such that $t_u u \in \mathcal{N}$.
\end{lemma}

\begin{proof}
	Let $u\in H\setminus \{0\}$. We define 
	\begin{equation}\label{2.7}
	\gamma_1(s)=s^2\|u\|^2+bs^4|\nabla u|_2^4+2bs^4|\nabla u|^2_2\int_{\mathbb{R}^N}V(x)u^2d x-\int_{\mathbb{R}^N}f(x,su)sud x.
	\end{equation}
	Using $(f2)$-$(f3)$, we can see that $	\gamma_1(0)=0$, $	\gamma_1(t)<0$ for $t>0$ large and $	\gamma_1(t)>0$ for $t>0$ small. Since $\gamma_1$ is continuous, there exists $t_u>0$ such that $	\gamma_1(t_u)=0$. We know that $t_u$ is the unique root of $\gamma_1(t)$. Indeed, if there exist another $\tilde{t}_u>0$ root, then
	\[
	\gamma_1(t_u)=\gamma_1(\tilde{t}_u)=0,
	\]
	and so, by \eqref{d2.2}, 
	\[
	J'(t_u u)t_u u=J'(\tilde{t}_u u)\tilde{t}_u u=0
   \]
   which together with \eqref{2.3} implies 
   \begin{align*}\label{2.10}
   J(t_u u)&\geq J(\tilde{t}_u u)+\frac{1-(\tilde{t}_u/t_u)^4}{4}J'(t_u u)t_u u+\frac{(1-\mu)(1-(\tilde{t}_u/t_u)^2)^2}{4}\|t_u u\|^2\\
   &=J(\tilde{t}_u u)+\frac{(1-\mu)(1-(\tilde{t}_u/t_u)^2)^2}{4}\|t_u u\|^2
   \end{align*}
   and
    \begin{align*}
   J(\tilde{t}_u u)&\geq J({t}_u u)+\frac{1-({t}_u/\tilde{t}_u)^4}{4}J^\prime(\tilde{t}_u u)\tilde{t}_u u+\frac{(1-\mu)(1-({t}_u/\tilde{t}_u)^2)^2}{4}\|\tilde{t}_u u\|^2\\
   &=J({t}_u u)+\frac{(1-\mu)(1-({t}_u/\tilde{t}_u)^2)^2}{4}\|\tilde{t}_u u\|^2.
   \end{align*}
   Then, comparing the above expressions, we have
   \[t_u=\tilde{t}_u.\]
   So, there exists unique $t_u$ such that $\gamma_1(t_u)=0$, for any $u\in H\setminus\{0\}$, namely, $t_u u\in \mathcal{N}$.
\end{proof}

\begin{lemma}\label{lemma2.5}
	Assume that $(f1)$-$(f4)$ are satisfied. Then
	\[
	\inf_{u\in \mathcal{N}}J(u)=c_{\mathcal{N}}=\inf_{u\in H\setminus\{0\}}\max_{t\geq 0}J(tu).
	\]
\end{lemma}
\begin{proof}
	Firstly, from \eqref{2.6}, we obtain
	\[
	c_{\mathcal{N}}=	\inf_{u\in \mathcal{N}}J(u)=	\inf_{u\in \mathcal{N}}\max_{t\geq 0}J(tu)\geq \inf_{u\in H\setminus\{0\}}\max_{t\geq 0}J(tu).
	\]
	Finally, for $u\in H\setminus\{0\} $, it follows from Lemma \ref{lemma2.4} that 
\begin{align}\label{2.12}
	c_{\mathcal{N}}=\inf_{z\in \mathcal{N}}J(z)\leq  J(t_uu)\leq \max_{t\geq 0}J(t u),
\end{align}
and so,
\begin{align}\label{2.13}
	c_{\mathcal{N}}=\inf_{z\in \mathcal{N}}J(z)\leq \inf_{u\in H\setminus\{0\}}\max_{t\geq 0}J(t_uu).
\end{align}
\end{proof}

\begin{lemma}\label{lemma2.7}
Assume that $(f1)$-$(f4)$ are satisfied. Then
	\[
		c_{\mathcal{N}}>0.
	\]
\end{lemma}
\begin{proof}
	If $u\in \mathcal{N}$, then $J'(u)u=0$ and by $(f1),$ $(f2)$ and Sobolev embedding theorem, one has 
	\begin{align*}
	\|u\|^2\leq \|u\|^2+b|\nabla u|^4_2+2b|\nabla u|^2_2\int_{\mathbb{R}^N}V(x)u^2d x&=\int_{\mathbb{R}^N}f(x,u)ud x\\
	&\leq \frac{1}{2}\int_{\mathbb{R}^{N}}\inf_{\mathbb{R}^N}V(x)u^2 d x+c\|u\|^p_p\\
	&\leq \frac{1}{2}\|u\|^2+c\|u\|^p
	\end{align*}
	and so,
	\[
	\|u\|\geq \hat{c}>0	
	\]
	for some $\hat{c}>0$.
	By \eqref{2.14}, we get
	\[
	J(u)=J(u)-\frac{1}{4}J'(u)u=\frac{1-\mu}{4}C>0.
	\]
	This implies $	c_{\mathcal{N}}\geq \frac{1-\mu}{4}C>0$.
\end{proof}

\begin{lemma}\label{lemma2.8}
	Assume that $(f1)$-$(f4)$ are satisfied.Then there exist some constant $d\in (0,c_{\mathcal{N}}] $ and a sequence $\{u_n\}\subset H$ such that
	\begin{equation}\label{2.16}
	J(u_n)\to d,\quad \|J'(u_n)\|(1+\|u_n\|)\to 0.
	\end{equation}
\end{lemma}
\begin{proof}
	By $(f1),$ $(f2)$ and \eqref{funcj}, for $u\in H$ we have that there exist $\rho >0$ and $\eta>0$ such that letting $\|u\|=\rho$ be small enough, we get $J(u)\geq\eta$. Let $w_k\in \mathcal{N}$ such that, for each $k\in \mathbb{N}$, we have
	\begin{equation}\label{2.18}
c_{\mathcal{N}}+\frac{1}{k}> J(w_k)\geq c_{\mathcal{N}}.
	\end{equation}
	By $J(tw_k)<0$ for large $t>0$ and \eqref{2.18}, we can use Moutain pass Lemma to verify that there exist a sequence $\{u_{k,n}\}\subset H$ such that
		\begin{equation}\label{2.19}
	J(u_{k,n})\to d_k,\quad \|J'(u_{k,n})\|(1+\|u_{k,n}\|)\to 0,
	\end{equation}
	where $d_k\in [\eta, \sup_{t\geq 0}J(t w_k)]$. From \eqref{2.5}, one has 
	\begin{equation*}
	J(w_k)\geq J(tw_k),\quad t\geq 0,
	\end{equation*}
	and so,
	\[
	J(w_k)=\sup_{t\geq 0}J(tw_k).
	\]
	Thus, by \eqref{2.18} and \eqref{2.19}, one has
		\begin{equation}\label{2.20}
	J(u_{k,n})\to d_k<c_{\mathcal{N}}+\frac{1}{k},\quad \|J'(u_{k,n})\|(1+\|u_{k,n}\|)\to 0.
	\end{equation}
	From \eqref{2.20}, if $k=1$ we get $n_1>0$ such that 
	\begin{equation*}
	    	J(u_{1,n_1})\to d_1<c_{\mathcal{N}}+{1},\quad \|J'(u_{1,n_1})\|(1+\|u_{1,n_1}\|)<1;
	\end{equation*}
	if $k=2$ there exist $n_2>n_1>0$ such that 
		\begin{equation*}
	    	J(u_{2,n_2})\to d_2<c_{\mathcal{N}}+
	    	\frac{1}{2},\quad \|J'(u_{2,n_2})\|(1+\|u_{2,n_2}\|)<\frac{1}{2}.
	\end{equation*}
Actually, we can get a sequence $n_k\rightarrow \infty$ as $k\rightarrow \infty$ and there exist a sequence $\{u_{k,n_k}\}\subset H$ satisfying
		\begin{equation}\label{2.21.0}
	J(u_{k,n_k})< c_{\mathcal{N}}+\frac{1}{k},\quad \|J'(u_{k,n_k})\|(1+\|u_{k,n_k}\|)<\frac{1}{k}.
	\end{equation}
	Therefore, going if necessary to a subsequence, by virtue of \eqref{2.21.0}, one has
	\[
	J(u_{n})\to d\in [\eta, c_{\mathcal{N}}],\quad \|J^\prime(u_n)\|(1+\|u_{n}\|)\to 0.
	\]
\end{proof}

\begin{lemma}\label{lemma1.10}
    The sequence $\{u_n\}$ is bounded in $H.$
\end{lemma}
\begin{proof}
By \eqref{2.16}, we have
\begin{align*}
    d+o_{n}(1)&=J(u_n)-\frac{1}{4} J^{\prime}\left(u_{n}\right) u_{n}\\
    &\geq\left(\frac{1}{2}-\frac{1}{4}\right)\|u_n\|^2-\int_{\mathbb{R}^{N}}\frac{\mu a}{4}V(x)u_n^2d x\\
    &\geq \frac{1}{4}\|u_n\|^2-\frac{\mu}{4}\|u_n\|^2\\
    &=\frac{(1-\mu)}{4}\|u_n\|^2.
\end{align*}
This shows that $\{u_n\}$ is bounded.
\end{proof}

\begin{lemma}\label{2.9}
	Assume that $(f1)$-$(f4)$ are satisfied. Since $\{u_n\}$ is bounded in $H$, then there exists $\tilde{u}\in H$ such that $J^\prime(\tilde{u})=0$. Moreover, if $\tilde{u}\neq 0$, going if necessary to a subsequence, then 
	\[\int_{\mathbb{R}^{N}}\left|\nabla u_{n}\right|^{2} d x \rightarrow \int_{\mathbb{R}^{N}}\left|\nabla\tilde{u}\right|^{2} d x,\quad \mbox{as}\quad n \rightarrow \infty\]
	and
	\[\int_{\mathbb{R}^{N}}\left|\nabla u_{n}\right|^{2} +2V(x)u_n^2 d x \rightarrow \int_{\mathbb{R}^{N}}\left|\nabla\tilde{u}\right|^{2}+ 2V(x)\tilde{u}^2d x,\quad \mbox{as}\quad n \rightarrow \infty.\]
	
\end{lemma}

\begin{proof}
Since $\{u_n\}$ is bounded in $H$ and $H$ is a reflexive Banach space, there exists $\tilde{u}\in H$ such that
	\begin{equation}\label{2.22.0}
	\left\{\begin{array}{l}{u_{n} \rightharpoonup \tilde{u} \text { in } H^{2}\left(\mathbb{R}^{N}\right)} \\ {u_{n} \rightarrow\tilde{u} \text { in } L_{loc}^{q}\left(\mathbb{R}^{N}\right)\left(2 \leq q<2^{*}\right)},~2^*=2N/(N-2) \\ {u_{n}(x) \rightarrow\tilde{u}(x) \text { a.e. on } \mathbb{R}^{N} \text { . }}\end{array}\right.
	\end{equation}
	If $\tilde{u}=0$, then $J^\prime(\tilde{u})\tilde{u}=0$. Now, if $\tilde{u}\neq 0$, up to a subsequence, there are $C_1>0$ and $C_2>0$ such that
	\[\int_{\mathbb{R}^{N}}\left|\nabla u_{n}\right|^{2} d x \rightarrow C_1^2,\quad \mbox{as}\quad n \rightarrow \infty\]
	and
	\[\int_{\mathbb{R}^{N}}\left|\nabla u_{n}\right|^{2} +2V(x)u_n^2 d x \rightarrow C_2^2,\quad \mbox{as}\quad n \rightarrow \infty.\]
Since $u_{n} \rightharpoonup\tilde{u}$ in $H$, by Lemma 2 in \cite{20imb}, we get 
	\[\int_{\mathbb{R}^{N}}\left|\nabla\tilde{u}\right|^{2} d x \leq \liminf_{n}\int_{\mathbb{R}^{N}}\left|\nabla u_{n}\right|^{2} d x =C_1^2 ,\quad \mbox{as}\quad n \rightarrow \infty\]
	and
	\[   \int_{\mathbb{R}^{N}}\left|\nabla\tilde{u}\right|^{2} +2V(x)\tilde{u}^2 d x  \leq \liminf_{n} \int_{\mathbb{R}^{N}}\left|\nabla u_{n}\right|^{2} +2V(x)u_n^2 d x= C_2^2,\quad \mbox{as}\quad n \rightarrow \infty.\]
	
	We argue by contradiction. Suppose that 
	\begin{equation}\label{pri}
	\int_{\mathbb{R}^{N}}\left|\nabla \tilde{u}\right|^{2} d x<C_1^2
	\end{equation}
	and
	\begin{equation}\label{segun}
	 \int_{\mathbb{R}^{N}}\left|\nabla\tilde{u}\right|^{2} +2V(x)\tilde{u}^2 d x < C_2^2.
    \end{equation}
	Let $\psi \in C^\infty_0(\mathbb{R}^N)$, by \eqref{2.16}, we get
	\begin{align}\label{2.23.0}
	    \lim_{n}J^\prime(u_n)\psi&=(\tilde{u},\psi)+b\left(C_2^2\right)\int_{\mathbb{R}^N}\nabla \tilde{u} \nabla \psi d x  +b(C_1^2)\int_{\mathbb{R}^N}V(x)\tilde{u}\psi d x \nonumber\\
	    &-\int_{\mathbb{R}^N}f(x,\tilde{u})\psi d x=0.
	\end{align}
	By approximation \eqref{2.23.0} is satisfied for all $\psi \in H$. Thus, 
	\begin{align}\label{2.24}
	   \|\tilde{u}\|^2+b\left(C_2^2\right)\int_{\mathbb{R}^N}|\nabla \tilde{u}|^2 d x  +b(C_1^2)\int_{\mathbb{R}^N}V(x)\tilde{u}^2 d x -\int_{\mathbb{R}^N}f(x,\tilde{u})\tilde{u} d x=0.
	\end{align}
	Then, if \eqref{pri} or \eqref{segun} occur, we get $J^\prime (\tilde{u})\tilde{u}<0$. From Lemma \ref{lemma2.4}, there exists $\tilde{t}>0$ such that $\tilde{t} \tilde{u}\in \mathcal{N}$. Therefore, $J(\tilde{t} \tilde{u})\geq c_{\mathcal{N}}$ and so, by Fatou's lemma and \eqref{2.3}, we have
	\begin{align*}
	\begin{aligned} c_{\mathcal{N}} & \geq d=\lim _{n \rightarrow \infty}\left(J\left(u_{n}\right)-\frac{1}{4} J^{\prime}\left(u_{n}\right) u_{n}\right) \\
	=& \lim _{n \rightarrow \infty}\left[\frac{1}{4}\left\|u_{n}\right\|^{2}+\int_{\mathbb{R}^{N}}\left(\frac{1}{4} f\left(x, u_{n}\right) u_{n}-F\left(x, u_{n}\right)\right) d x \right]\\
	 \geq & \frac{1}{4} \liminf _{n \rightarrow \infty}\left(\left\|u_{n}\right\|^{2}-\mu \int_{\mathbb{R}^{N}}a V(x) u_{n}^{2} d x\right) \\
	 &\quad+\liminf _{n \rightarrow \infty} \int_{\mathbb{R}^{N}}\left(\frac{1}{4} f\left(x, u_{n}\right) u_{n}-F\left(x, u_{n}\right)+\frac{\mu aV(x)}{4}u_{n}^{2}\right) {d}x\\
    \geq& \frac{1}{4}\left(\left\|\tilde{u}\right\|^{2}-\mu \int_{\mathbb{R}^{N}}a V(x)\tilde{u}^{2} d x\right)+\int_{\mathbb{R}^{N}}\left[\frac{1}{4} f\left(x,\tilde{u}\right)\tilde{u}-F\left(x,\tilde{u}\right)+\frac{\mu aV(x)}{4}\tilde{u}^{2}\right] d x\\
    =&\left(J\left(\tilde{u}\right)-\frac{1}{4} J^{\prime}\left(\tilde{u}\right)\tilde{u}\right)\\
    \geq &\left(J\left(\tilde{t}\tilde{u}\right)+\frac{1-{\tilde{t}}^4}{4}J^\prime(\tilde{u})\tilde{u}+(1-\mu)\frac{(1-{\tilde{t}}^2)^2}{4}\|\tilde{u}\|^2\right)-\frac{1}{4} J^{\prime}\left(\tilde{u}\right)\tilde{u}\\
    \geq& c_{\mathcal{N}}-\frac{{\tilde{t}}^4}{4}J^\prime(\tilde{u})\tilde{u}\\
    > & c_{\mathcal{N}}.
	\end{aligned}
	\end{align*}
	Hence, $J^\prime(\tilde{u})\tilde{u}=0$, and up to a subsequence, 
	\[ \lim_{n \rightarrow \infty}\int_{\mathbb{R}^{N}}\left|\nabla u_{n}\right|^{2} d x =\int_{\mathbb{R}^{N}}\left|\nabla\tilde{u}\right|^{2} d x ,\quad \mbox{as}\quad n \rightarrow \infty\]
	and
	\[ \lim_{n \rightarrow \infty} \int_{\mathbb{R}^{N}}\left|\nabla u_{n}\right|^{2} +2V(x)u_n^2 d x=  \int_{\mathbb{R}^{N}}\left|\nabla\tilde{u}\right|^{2} +2V(x)\tilde{u}^2 d x,\quad \mbox{as}\quad n \rightarrow \infty.\]
\end{proof}

Next, we prove the minimizer of the constrained problem is a critical point, which plays
a crucial role in the asymptotically periodic case.
\begin{lemma}\label{lemma2.10}
    Assume that $({V})$ and $({f} 1)-({f} 4)$ are satisfied. If $u_{0} \in \mathcal{N}$ and $J\left(u_{0}\right)=c_\mathcal{N},$ then $u_{0}$ is a critical point of $J.$
\end{lemma}
\begin{proof}
Let $u_{0} \in \mathcal{N}$, $J(u_0)=c_\mathcal{N}$ and $J^\prime(u_0)\not = 0.$  Then there exist $\delta>0$ and $\rho>0$ such that 

\begin{equation}\label{2.20.1}
\left\|u-u_{0}\right\| \leq 3 \delta \Rightarrow\left\|J^{\prime}(u)\right\| \geq \rho.
\end{equation}
By Lemma \ref{lemma2.1}, we have 
\begin{equation}\label{2.21}
\begin{aligned} J\left(t u_{0}\right) & \leq J\left(u_{0}\right)-\frac{\left(1-\mu\right)\left(1-t^{2}\right)^{2}}{4}\left\|u_{0}\right\|^{2} \\ &=c_\mathcal{N}-\frac{\left(1-\mu\right)\left(1-t^{2}\right)^{2}}{4}\left\|u_{0}\right\|^{2}, \quad \forall t \geq 0 \end{aligned}
\end{equation}
For $\varepsilon :=\min \left\{3\left(1-\mu\right)\left\|u_{0}\right\|^{2} / 64,1, \rho \delta / 8\right\}, S:=B\left(u_{0}, \delta\right),$ from \cite[Lemma 2.3]{willem} we get a deformation $\eta \in \mathcal{C}([0,1] \times H, H)$ such that\\
(i) $\eta(1, u)=u$ if $u\not \in J^{-1}([c_\mathcal{N}-2 \varepsilon,c_\mathcal{N}+2 \varepsilon]),$\\
(ii) $\eta\left(1, J^{c_\mathcal{N}+\varepsilon} \cap B\left(u_{0}, \delta\right)\right) \subset J^{c_\mathcal{N}-\varepsilon},$\\
(iii) $J(\eta(1, u)) \leq J(u), \forall u \in H,$\\
(iv) $\eta(1, u)$ is a homeomorphism of $H.$\\
By Corollary \ref{2.2} and (ii), one has
\begin{equation}\label{2.22}
J\left(\eta\left(1, t u_{0}\right)\right) \leq c_{\mathcal{N}}-\varepsilon, \quad \forall t \geq 0,|t-1|<\delta /\left\|u_{0}\right\|.
\end{equation}
Now, using \eqref{2.21} and (iii), we have that 
\begin{equation}\label{2.23}
\begin{aligned} J\left(\eta\left(1, t u_{0}\right)\right) & \leq J\left(t u_{0}\right) \\ & \leq c_\mathcal{N}-\frac{\left(1-\mu\right)\left(1-t^{2}\right)^{2}}{4}\left\|u_{0}\right\|^{2} \\ & \leq c_\mathcal{N}-\frac{\left(1-\mu\right) \delta^{2}}{4}, \quad \forall t \geq 0,|t-1| \geq \delta /\left\|u_{0}\right\| .\end{aligned}
\end{equation}
By \eqref{2.22} and \eqref{2.23}, it follows that
\begin{equation}\label{2.24.1}
\max _{t \in[1 / 2, \sqrt{7} / 2]} J\left(\eta\left(1, t u_{0}\right)\right)<c_\mathcal{N}.
\end{equation}
Let us to prove that $\eta\left(1, t u_{0}\right) \cap \mathcal{N} \neq \emptyset$ for some $t \in[1 / 2, \sqrt{7} / 2],$  which is a contradiction
with the definition of $c_\mathcal{N}.$ Set
$$
\sigma_{0}(t) := J^{\prime}\left(t u_{0}\right) t u_{0}, \quad \sigma_{1}(t) :=J^{\prime}(\eta\left(1, t u_{0}\right)) \eta\left(1, t u_{0}\right), \quad \forall t \geq 0
$$

By (iv), since $u_0\not=0$, one has $\eta\left(1, t u_{0}\right)$ for all $t> 0.$ From \eqref{2.21} and (i), it follows that $\eta\left(1, t u_{0}\right)=t u_{0}$ for $t=1 / 2$ and $t=\sqrt{7}/2$. On the other hand, Lemma \ref{lemma2.4} and degree theory implies $
\operatorname{deg}\left(\sigma_{0},(1 / 2, \sqrt{7} / 2), 0\right)=1$. Then, by the invariance of the degree for functions coinciding at the domain boundary,
$$
\operatorname{deg}\left(\sigma_{1},(1 / 2, \sqrt{7} / 2), 0\right)=\operatorname{deg}\left(\sigma_{0},(1 / 2, \sqrt{7} / 2), 0\right)=1.
$$
Thus there exists $t_{0} \in(1 / 2, \sqrt{7} / 2)$ such that $\sigma_1(t_0)=0$ which implies $\eta\left(1, t_0 u_{0}\right)\in \mathcal{N}$ and the proof is completed.
\end{proof}
\section{The Periodic Case}

\textit{Proof of Theorem \eqref{thm1.1}} Using Lemma \ref{lemma2.8}, we get a sequence $\{u_n\}\subset H$ that satisfies 
	\begin{equation}\label{3.1}
	J(u_n)\to d,\quad J'(u_n)u_n \to 0.
	\end{equation}
By \eqref{2.14} and \eqref{3.1}, for large $n\in \mathbb{N}$, we get
\[
d+1\geq J(u_n)-\frac{1}{4}J^\prime(u_n)u_n\geq \frac{1-\mu}{4}\|u_n\|^2.
\]

Then there exists $c>0$ such that $|u_n|^2_2\leq c$. If 
$$
l=\sup _{y \in \mathbb{R}^{N}} \int_{B_1(y)}\left|u_{n}\right|^{2} \rightarrow 0, n \rightarrow \infty,
$$

then, by  Lemma 1.21 \cite{willem}, one has $u_{n} \rightarrow 0$ in $L^{p}\left(\mathbb{R}^{N}\right)$ for $2<p<2_*$. By $(f1)$-$(f2),$ we get

\begin{align*}
\begin{aligned} d &=J\left(u_{n}\right)-\frac{1}{2} J^{\prime}\left(u_{n}\right) u_{n}+o(1) \\ &=-\frac{b}{4}\left|\nabla u_{n}\right|_{2}^{4}+\int_{\mathbb{R}^{3}}\left[\frac{1}{2} f\left(x, u_{n}\right) u_{n}-F\left(x, u_{n}\right)\right] d x+o_{n}(1) \\ & \leq o_{n}(1)+\varepsilon, \end{aligned}
\end{align*}
for any $\varepsilon>0$. Thus, $l>0$ and so, we may assume that there exist $\{y_n\}\in\mathbb{Z}^N$ such that 
\[
\int_{B_{1+\sqrt{N}}\left(y_{n}\right)}\left|u_{n}\right|^{2} d x>\frac{l}{2}.
\]
Let us define $v_n(x)=u_n(x+y_n)$, such that $\|v_n\|=\|u_n\|$,
\[
\int_{B_{1+\sqrt{N}}\left(0\right)}\left|v_{n}\right|^{2} d x>\frac{l}{2}
\]
and 
\[
J(v_n)\to d,\quad \|J^\prime(v_n)v_n\|(1+\|v_n\|)\to 0.
\]
Analogously, we may assume there exists $\tilde{v}\in H$ such that 
\begin{equation*}
	\left\{\begin{array}{l}{v_{n}\rightharpoonup \tilde{v} \text { in } H^{2}\left(\mathbb{R}^{N}\right)} \\ {v_{n} \rightarrow \tilde{v} \text { in } L_{loc}^{q}\left(\mathbb{R}^{N}\right)\left(2 \leq q<2_{*}\right)} \\ {v_{n}(x) \rightarrow \tilde{v}(x) \text { a.e. on } \mathbb{R}^{N} \text { . }}\end{array}\right.
	\end{equation*}
	Also, up to a subsequence,
	\[\int_{\mathbb{R}^{N}}\left|\nabla v_{n}\right|^{2} d x \rightarrow \int_{\mathbb{R}^{N}}\left|\nabla \tilde{v}\right|^{2} d x,\quad \mbox{as}\quad n \rightarrow \infty\]
	and
	\[\int_{\mathbb{R}^{N}}\left|\nabla v_{n}\right|^{2} +2V(x)v_n^2 d x \rightarrow \int_{\mathbb{R}^{N}}\left|\nabla \tilde{v}\right|^{2}+ 2V(x)\tilde{v}^2d x,\quad \mbox{as}\quad n \rightarrow \infty.\]
	We obtain
	\[
	J^\prime(\tilde{v})\psi=\lim_n J^\prime(v_n)\psi=0,\quad \forall \quad \psi \in H,
	\]
	which implies $J^\prime (\tilde{v})=0$ with $\tilde{v}\in \mathcal{N}$. Follows from \eqref{2.14}, Fatou's lemma and weak semicontinuity of norm that
		\begin{align*}
	\begin{aligned} c_{\mathcal{N}} & \geq d=\lim _{n \rightarrow \infty}\left(J\left(v_{n}\right)-\frac{1}{4} J^{\prime}\left(v_{n}\right) v_{n}\right) \\
	=& \lim _{n \rightarrow \infty}\left[\frac{1}{4}\left\|v_{n}\right\|^{2}+\int_{\mathbb{R}^{N}}\left(\frac{1}{4} f\left(x, v_{n}\right) v_{n}-F\left(x, v_{n}\right)\right) d x \right]\\
	 \geq & \frac{1}{4} \liminf _{n \rightarrow \infty}\left(\left\|v_{n}\right\|^{2}-\mu \int_{\mathbb{R}^{N}}a V(x) v_{n}^{2} d x\right) \\
	 &\quad+\liminf _{n \rightarrow \infty} \int_{\mathbb{R}^{N}}\left(\frac{1}{4} f\left(x, v_{n}\right) v_{n}-F\left(x, v_{n}\right)+\frac{\mu aV(x)}{4}v_{n}^{2}\right) {d}x\\
    \geq& \frac{1}{4}\left(\left\|\tilde{v}\right\|^{2}-\mu \int_{\mathbb{R}^{N}}a V(x) \tilde{v}^{2} d x\right)+\int_{\mathbb{R}^{N}}\left[\frac{1}{4} f\left(x, \tilde{v}\right) \tilde{v}-F\left(x, \tilde{v}\right)+\frac{\mu aV(x)}{4}\tilde{v}^{2}\right] d x\\
    =&\left(J\left(\tilde{v}\right)-\frac{1}{4} J^{\prime}\left(\tilde{v}\right)\tilde{v}\right).
	\end{aligned}
	\end{align*}
	Hence, $J(\tilde{v})= c_{\mathcal{N}}>0$ and $\tilde{v}\neq 0.$
	
	\section{The asymptotically periodic case}
	In this section, we have  $V(x)=V_{0}(x)+V_{1}(x)$ and  $f(x, u)=f_{0}(x, u)+f_{1}(x, u)$ \\
	 Define functional $J_{0}$ as follows: 
	 \begin{equation}\label{4.1}
J_0(u)=\frac{1}{2}\left[\int_{\mathbb{R}^{N}}\left(a|\nabla u|^{2}+V_0(x) u^{2}\right) d x\right]+\frac{b}{4}|\nabla u|^4_2+\frac{b}{2}(|\nabla u|^2_2)\int_{\mathbb{R}^N}V(x)u^2d x-\int_{\mathbb{R}^{N}} F_0(x, u) d x
\end{equation}
where $F_{0}(x, u) :=\int_{\mathbb{R}^{N}} f_{0}(x, s) d s$. By $(V'),$ $(f1),$ $(f2),$ $(f5)$ and $(f6)$ we have $J_{0} \in \mathcal{C}^{1}(H, \mathbb{R})$ and
\begin{equation}\label{4.2}
J_0^{\prime}(u) v=(u, v)+b\left(|\nabla u|_{2}^{2}+\int_{\mathbb{R}^{N}} V_0(x) u^{2} d x\right) \int_{\mathbb{R}^{N}} \nabla u \nabla v d x+b\left(|\nabla u|_{2}^{2}\right) \int_{\mathbb{R}^{N}} V_0(x) u v d x-\int_{\mathbb{R}^{N}} f_0(x, u) v d x
\end{equation}

\begin{lemma}
    Assume that $(V'),$ $(f1),$ $(f2),$ $(f5)$ and $(f6)$ are satisfied. Then, if $u_{n} \rightharpoonup 0$ in $H$, we have
    \begin{equation}\label{4.3}
\lim _{n \rightarrow \infty} \int_{\mathbb{R}^{N}} V_{1}(x) u_{n}^{2} d x=0, \quad \lim _{n \rightarrow \infty} \int_{\mathbb{R}^{N}} V_{1}(x) u_{n} v d x=0, \quad \forall v \in H;
\end{equation}

\begin{equation}\label{4.4}
\lim _{n \rightarrow \infty} \int_{\mathbb{R}^{N}} F_{1}\left(x, u_{n}\right) d x=0, \quad \lim _{n \rightarrow \infty} \int_{\mathbb{R}^{N}} f_{1}\left(x, u_{n}\right) v d x=0, \quad \forall v \in H.
\end{equation}
\end{lemma}

\noindent
\textit{Proof of Theorem \ref{thm1.2}}. Lemma \ref{lemma2.8} implies the existence of a sequence $\{u_n\}$ in $H$ such that 
\begin{equation}\label{4.5}
J\left(u_{n}\right) \rightarrow d, \qquad\left\|J^{\prime}\left(u_{n}\right)\right\|\left(1+\left\|u_{n}\right\|\right) \rightarrow 0.
\end{equation}
\noindent
By Lemma \ref{lemma1.10}, one has $\{u_n\}$ bounded and then, up to a subsequence, $u_n\rightharpoonup u$ for some $u\in H.$ Hence, 
	\begin{equation*}
	\left\{\begin{array}{l}{u_{n} \rightharpoonup u \text { in } H^{2}\left(\mathbb{R}^{N}\right)} \\ {u_{n} \rightarrow u \text { in } L_{loc}^{q}\left(\mathbb{R}^{N}\right)\left(2 \leq q<2_{*}\right)} \\ {u_{n}(x) \rightarrow u(x) \text { a.e. on } \mathbb{R}^{N} \text { . }}\end{array}\right.
	\end{equation*}
	Similarly to the proof of Theorem \ref{thm1.1}, if $u=0$, then
		\begin{equation*}
	\left\{\begin{array}{l}{u_{n} \rightharpoonup 0 \text { in } H^{2}\left(\mathbb{R}^{N}\right)} \\ {u_{n} \rightarrow 0 \text { in } L_{loc}^{q}\left(\mathbb{R}^{N}\right)\left(2 \leq q<2_{*}\right)} \\ {u_{n}(x) \rightarrow 0 \text { a.e. on } \mathbb{R}^{N} \text { . }}\end{array}\right.
	\end{equation*}
 Observe that
\begin{equation}\label{4.6}
    \|u\|^2=\int_{\mathbb{R}^N}(|\Delta u|^2+a(|\nabla u|^2+V_0(x)u^{2}) )d x+\int_{\mathbb{R}^N}V_1(x)u^2d x,\quad \forall u\in H;
\end{equation} 
	
\begin{equation}\label{4.7}
    J_{0}(u)=J(u)-\frac{a}{2} \int_{\mathbb{R}^{N}} V_{1}(x) u^{2} d x+\int_{\mathbb{R}^{N}} F_{1}(x, u) d x, \quad \forall u \in H
\end{equation}	
	and
	\begin{equation}\label{4.8}
	    J_0^{\prime}(u) v= J^{\prime}(u)v-a\int_{\mathbb{R}^{N}} V_{1}(x) u v d x+\int_{\mathbb{R}^{N}} f_{1}(x, u) v d x, \quad \forall u, v \in H.
	\end{equation}
	By \eqref{2.16}, \eqref{4.3}-\eqref{4.5}, \eqref{4.7}-\eqref{4.8}, one has
	\begin{equation}\label{4.9}
J_{0}\left(u_{n}\right) \rightarrow d, \qquad\left\|J_{0}^{\prime}\left(u_{n}\right)\right\|\left(1+\left\|u_{n}\right\|\right) \rightarrow 0.
\end{equation}
	As in the proof of Theorem \ref{thm1.1}, there exists $y_{n} \in \mathbb{Z}^{N},$ up to a subsequence, such that
	\begin{equation}\label{4.10}
\int_{B_{1+\sqrt{N}}\left(y_{n}\right)}\left|u_{n}\right|^{2} d x>\frac{l}{2}
    \end{equation}
Let us define $v_n(x)=u_n(x+y_n)$, such that $\|v_n\|=\|u_n\|$,
\[
\int_{B_{1+\sqrt{N}}\left(0\right)}\left|v_{n}\right|^{2} d x>\frac{l}{2}
\]
and 
\begin{equation}\label{4.111}
J_0(v_n)\to d\in (0,	c_{\mathcal{N}}],\quad \|J_0^\prime(v_n)\|(1+\|v_n\|)\to 0.
\end{equation}
Up to a subsequence, we have
$$
\left\{\begin{array}{l}v_n \rightharpoonup v_0\text { in } H^{2}(\mathbb{R}^{N}) \\
{v_{n} \rightarrow v_{0} \text { in } L_{l o c}^{q}\left(\mathbb{R}^{N}\right)\left(2 \leq q<2^{*}\right)} \\
{v_{n}(x) \rightarrow v_{0}(x) \text { a.e. on } \mathbb{R}^{N}}\end{array}\right.
$$
From \eqref{4.10}, we conclude that $v_0\not = 0.$ In view of \eqref{2.3}, Corollary \ref{2.2}, Lemma \ref{lemma2.5}, \eqref{4.7} and \eqref{4.8}, we obtain
\begin{equation}\label{4.13}
J_{0}(u)=\max _{t \geq 0} J_{0}(t u), \quad \forall u \in \mathcal{N}_{0}, \quad \inf _{u \in \mathcal{N}_0} J_{0}(u)=c_{\mathcal{N}_0}=\inf _{u \in H\setminus \{0\}} \max _{t \geq 0} J_{0}(t u)>0,
\end{equation}
where
$$
\mathcal{N}_{0} :=\left\{u \in H :u \neq 0, J_{0}^{\prime}(u) u= 0\right\}.
$$
From Theorem \ref{thm1.1} there exists $v_0\in \mathcal{N}_{0}$ such that $J_0(u_0)=c_{\mathcal{N}_0}>0.$ By $(V'),$ $(f5),$ \eqref{4.7} and \eqref{4.13}, we obtain
\begin{equation}\label{4.14}
c_\mathcal{N}=\inf _{v \in \mathcal{N}} \max _{t \geq 0} J(t v) \leq \max _{t \geq 0} J\left(t v_{0}\right) \leq \max _{t \geq 0} J_{0}\left(t v_{0}\right) \leq J_{0}\left(v_{0}\right)=c_{\mathcal{N}_{0}}.
\end{equation}

By $(f5)$ and \eqref{4.8}, we have 
$$
J^{\prime}(v_0) v_0 \leq J_{0}^{\prime}(v_0)v_0= 0.
$$
From \eqref{2.4}, \eqref{2.14}, \eqref{4.1}-\eqref{4.2}, \eqref{4.111}, the weakly lower semi-continuity of the norm and
Fatou's lemma, we have
	\begin{align*}
	\begin{aligned} c_{\mathcal{N}} & \geq d=\lim _{n \rightarrow \infty}J_0\left(v_{n}\right)-\frac{1}{4} J_0^{\prime}\left(v_{n}\right) v_{n} \\
	=& \lim _{n \rightarrow \infty}\left[\frac{1}{4}\left\|v_{n}\right\|^{2}+\int_{\mathbb{R}^{N}}\left(\frac{1}{4} f_0\left(x, v_{n}\right) v_{n}-F_0\left(x, v_{n}\right)\right) d x \right]\\
	 \geq & \frac{1}{4} \liminf _{n \rightarrow \infty}\left(\left\|v_{n}\right\|^{2}-\mu \int_{\mathbb{R}^{N}}a V_0(x) v_{n}^{2} d x\right) \\
	 &\quad+\liminf _{n \rightarrow \infty} \int_{\mathbb{R}^{N}}\left(\frac{1}{4} f_0\left(x, v_{n}\right) v_{n}-F_0\left(x, v_{n}\right)+\frac{\mu aV_0(x)}{4}v_{n}^{2}\right) {d}x\\
    \geq& \frac{1}{4}\left(\left\|v_{0}\right\|^{2}-\mu \int_{\mathbb{R}^{N}}a V_0(x) v_{0}^{2} d x\right)+\int_{\mathbb{R}^{N}}\left[\frac{1}{4} f_0\left(x, v_{0}\right) v_{0}-F_0\left(x, v_{0}\right)+\frac{\mu aV_0(x)}{4}v_{0}^{2}\right] d x\\
    =&\left(J_0\left(v_{0}\right)-\frac{1}{4} J_0^{\prime}\left(v_{0}\right)v_{0}\right)\\
    =&J_0\left(v_{0}\right)
	\end{aligned}
	\end{align*}
and so, $c_\mathcal{N}\geq J_0\left(v_{0}\right).$
In view of the Lemma \ref{lemma2.4}, there exists $t_0>0$ such that $t_0v_0\in \mathcal{N}.$ Then $J\left(t_{0}v_0\right) \geq c_{\mathcal{N}}.$ In fact, $J\left(t_{0}v_0\right) = c_{\mathcal{N}}.$ Arguing by contradiction, suppose that $J\left(t_{0}v_0\right) > c_{\mathcal{N}},$ and so, by $(V'),$ $(f5),$ \eqref{2.6}, \eqref{4.7} and \eqref{4.8},
\begin{align*}
	 c_{\mathcal{N}}&\geq J_0(v_0)\geq J_0(t_0v_0)\\
	 &=J(t_0v_0)-\frac{a}{2} \int_{\mathbb{R}^{N}} V_{1}(x) (t_0v_0)^{2} d x+\int_{\mathbb{R}^{N}} F_{1}(x, t_0v_0) d x\\
	 &\geq J(t_0v_0)> c_{\mathcal{N}}.
\end{align*}
This shows $J\left(t_{0}v_0\right) = c_{\mathcal{N}}.$

Take $u_0=t_0v_0$ and so, from Lemma \ref{lemma2.10} we have $J'(u_0)=0$. Thus $u_0$ is a solution of \eqref{p} when $V$ and $f$ are asymptotically periodic. Finally, if $u\not =0$ we can argue as in the final part of Theorem \ref{thm1.1} to obtain $J(u)=c_{\mathcal{N}}>0$ and $u\in H$ is a nontrivial solution for \eqref{p}. 

\section{Conflicts of Interest}
	The author declare no conflicts of interest. 
	
\bigskip
\bigskip

\end{document}